\documentclass[12pt]{amsart}
\usepackage[utf8]{inputenc}
\usepackage[T1]{fontenc}
\usepackage[top=2.5cm, bottom=2.5cm, left=2.5cm, right=2.5cm]{geometry}
\usepackage{csquotes, amsfonts, amsmath, amsthm, amssymb, stmaryrd, dsfont, enumitem, xcolor, url, graphicx, multicol, float, setspace, hyperref, tikz, multirow, esvect, colortbl, pifont, doi, url}
\usepackage[normalem]{ulem}
\usepackage[breakable, skins]{tcolorbox}

\setlist[itemize]{leftmargin=1cm}
\setlist[enumerate]{leftmargin=1cm}

\theoremstyle{plain}
\newtheorem{theo}{Theorem}[section]
\newtheorem{prop}[theo]{Proposition}

\newtheorem{coro}[theo]{Corollary}
\newtheorem{lem}[theo]{Lemma}

\theoremstyle{definition}

\newtheorem{rem}[theo]{Remark}
\newtheorem{expl}[theo]{Example}

\hypersetup{
    colorlinks=true,
    linkcolor=black,
    urlcolor=black,
    citecolor=black
}

\newcommand{\N}{\mathbb{N}}

\newcommand{\R}{\mathbb{R}}
\newcommand{\C}{\mathbb{C}}
\newcommand{\D}{\mathbb{D}}
\newcommand{\Z}{\mathbb{Z}}
\newcommand{\T}{\mathbb{T}}
\newcommand{\Q}{\mathbb{Q}}

\newcommand{\A}{\mathbb{A}}
\renewcommand{\H}{\mathbb{H}}

\newcommand{\dd}{\mathrm{d}}
\newcommand{\Id}{\mathrm{Id}}

\newcommand{\Hol}{\mathop{\rm Hol}\nolimits}

\newcommand{\Ker}{\mathrm{Ker}}

\newcommand{\Cont}{\mathcal{C}}

\newcommand{\NN}{\mathbb{N}}

\newcommand\Talbe{T_{\alpha,\beta}}
\newcommand{\ZZ}{\mathbb Z}

\newcommand\spam{\mathop{\rm span}\nolimits} 

\newcommand{\norm}[1]{\left\Vert#1\right\Vert}

\newcommand{\abs}[1]{\left\lvert #1 \right\rvert}

\newcommand{\sbt}{\,\begin{picture}(-1,1)(-1,-3)\circle*{3}\end{picture}\ }

\renewcommand{\textbf}[1]{\begingroup\bfseries\mathversion{bold}#1\endgroup}

\begin{document}


\author{I. Chalendar}
\address{Isabelle CHALENDAR, Université Gustave Eiffel, LAMA, (UMR 8050), 
    UPEM, UPEC, CNRS, F-77454, Marne-la-Vallée (France)}
\email{isabelle.chalendar@univ-eiffel.fr}

\author{L. Oger}
\address{Lucas OGER, Université Gustave Eiffel, LAMA, (UMR 8050), 
    UPEM, UPEC, CNRS, F-77454, Marne-la-Vallée (France)}
\email{lucas.oger@univ-eiffel.fr}

\author{J. R. Partington}
\address{Jonathan R. PARTINGTON, School of Mathematics, University of Leeds, Leeds LS2 9JT, Yorkshire, U.K.}
\email{j.r.partington@leeds.ac.uk}

\title[Linear isometries on the annulus]
{Linear isometries on the annulus: description and spectral properties}

\keywords{Fréchet space; holomorphic functions; isometry; annulus; weighted composition operator; spectrum}
\subjclass{Primary 47B33; Secondary 30H05, 47A10}
\thanks{This research is partly supported by the Bézout Labex, funded by ANR, reference ANR-10-LABX-58.}

\begin{abstract}
    We give a complete characterisation of the linear isometries of $\Hol(\Omega)$, where $\Omega$ is the half-plane, the complex plane or an annulus centered at 0 and symmetric to the unit circle.
    Moreover, we introduce new techniques to describe the holomorphic maps on the annulus that preserve the unit circle, and we finish by proving results about the spectra of the linear isometries on the annulus.
\end{abstract}

\maketitle


\section{Introduction}\label{Sec - Introduction}

In order to study a metric linear space, and in particular its geometry, one of the main tools is  an analysis of the linear isometries of this space. For spaces of holomorphic functions, there is a strong connection between isometries and weighted composition operators. \medskip

This link can be traced back to Banach himself \cite{Banach} in the space $\Cont(K)$ of continuous real-valued functions on a compact metric space $K$. For most well-known spaces, such as the Hardy, Bergman, Bloch spaces or the disc algebra, the linear isometries have been completely described \cite{Chal-Part, Cima-Wogen, Colonna, El-G-W, Forelli, Kolaski, Schwartz, Zorboska}. \medskip

This paper can be considered as a sequel to \cite{COP}, where the authors characterised completely the linear isometries $V$ on the space $\Hol(\D)$ of holomorphic functions on the unit disc, relative to one of the two classical distances. We showed there that $V$ is of the form
\[ V(f)(z) = \alpha f(\beta z) =: T_{\alpha, \beta}(f)(z),
    \quad \alpha, \beta \in \T, \]
where $\T$ is the unit circle of $\C$. Those $V$ are determined by being isometric for any two seminorms of $\Hol(\D)$, defined by $\norm{f}_{\infty, r} := \sup_{z \in K_r} \abs{f(z)}$, with $r \in [0, 1)$ and $K_r = r\overline \D$. \smallskip

It is natural to consider the same problem when replacing the unit disc $\D$ with other domains, such as the half-plane $\H$ or the whole plane $\C$. Those two cases are easily obtained from the results of \cite{COP}.

\begin{itemize}[label=$\star$]
    \item For the half-plane $\H$, let us define the \emph{Cayley transform} $\tau : \D \to \H$ by
    \[ \tau(z) = \frac{1+z}{1-z}. \]
    The map $\tau$ is  biholomorphic. For all $r \in [0, 1)$, denote $L_r = \tau(K_r)$. We aim to find a linear isometry $V$ of $\Hol(\H)$ for the seminorms associated with the sets $L_r$. To do so, note that for all $F \in \Hol(\H)$, the map $f = F \circ \tau$ is in $\Hol(\D)$. Hence, for every linear isometry $V$ of $\Hol(\H)$,
    \begin{align*}
        \sup_{w \in K_r} \abs{f(z)}
        & = \sup_{z \in L_r} \abs{F(z)} \\
        & = \sup_{z \in L_r} \abs{(VF)(z)}
        = \sup_{w \in K_r} \abs{(C_\tau \circ V \circ C_{\tau^{-1}})(f)(w)}.
    \end{align*} 
    Therefore, $C_\tau \circ V \circ C_{\tau^{-1}}$ is a linear isometry of $\Hol(\D)$, so we obtain that $V$ is of the form $C_{\tau^{-1}} \circ T_{\alpha, \beta} \circ C_\tau$ for some $\alpha, \beta \in \T$. \medskip

    \item For the plane $\C$, assume that $\norm{Tf}_{\infty, \rho} = \norm{f}_{\infty, \rho}$ for all $f \in \Hol(\C)$ and $\rho \ge 0$. 
    Then, it is in particular true for all $f$ polynomial. By density, for all $f \in \Hol(\D)$ and $0 \le \rho < 1$, $\norm{Tf}_{\infty, \rho} = \norm{f}_{\infty, \rho}$. Thus, $T$ is an isometry of $\Hol(\D)$, and $T = T_{\alpha, \beta}$ for some $\alpha, \beta \in \T$.
\end{itemize}

We now consider a doubly-connected domain, that is an annulus.
We fix $R > 1$ and define $\A = \{z \in \C : 1/R < \abs{z} < R\}$. Our main goal is to obtain a similar characterisation as in \cite{COP}. \medskip
    
In Sections \ref{Sec - First results} and \ref{Sec - Characterisations}, we give preliminary results about distances, seminorms, composition operators and powers of $z$. These will be present in the proof of the Main Theorem (Section \ref{Sec - Main result}), which says that if we consider annuli symmetric to the unit circle, then the only linear isometries of $\Hol(\A)$ are rotations or inversions. \medskip

In \cite{COP}, we also used a description of the Blaschke products to obtain a characterisation of linear isometries for one seminorm associated to one of the $K_r$. In Section \ref{Sec - Unimodular maps unit circle}, we describe the annulus version of Blaschke products, that is, unimodular functions on the unit circle. Results of Fatou \cite{Fatou23} and Bourgain \cite{Bourgain86} will be in the spotlight in this section. \medskip

Finally, Section \ref{Sec - Spectrum} is devoted to the spectral study of the linear isometries of $\Hol(\A)$. In particular, depending on the angle of the rotation, we give examples of rotations such that the spectrum is different from the point spectrum or the unit circle.

\section{First results}\label{Sec - First results}

This section is devoted to proving the main ingredients in the proof of the main result, that is, the description of linear isometries on the Fréchet space of analytic functions on an annulus.

The content of this section may be of independent interest. 

\subsection{Algebraic characterisation  of composition operators}

From now on, for $n \in \Z$, we write $e_n$ for the function defined by $e_n(z)=z^n$.

\begin{lem}\label{Lemme - Formes linéaires}
    Let $L : \Hol(\A) \to \C$ be non-zero, linear, continuous and multiplicative. 
    Then there exists $z_0 \in \A$ such that for all $f \in \Hol(\A)$, $L(f) = f(z_0)$.
\end{lem}

\begin{proof}
    The proof uses  the same ideas as in \cite[Lemma 2.2]{ACC}. For the sake of completeness, we detail them.  
    Note that for all $f \in \Hol(\A)$, 
    \[ L(f) = L(f e_0) = L(f) L(e_0). \] 
    Choosing $f$ such that $L(f) \neq 0$, we have $L(e_0) = 1$. Denote by $z_0 = L(e_1)$. Thus,
    \begin{itemize}[label=\sbt]
        \item Assume that $z_0 \not\in \A$. 
        Then $g = (e_1 - z_0 e_0)^{-1} \in \Hol(\A)$ and 
        \[ (e_1 - z_0 e_0) g = e_0 
            \implies 0 = [L(e_1) - z_0 L(e_0)] L(g) = L(e_0) = 1, \] 
        a contradiction. Hence, $z_0 \in \A$. \medskip

        \item Let $f \in \Hol(\A)$ such that $f(z_0) = 0$. There exists $h \in \Hol(\A)$ such that $f = (e_1 - z_0 e_0) h$, so 
        \[ L(f) = [L(e_1) - z_0 L(e_0)]L(h) = [z_0 - z_0]L(h) = 0. \medskip \]

        \item If $f(z_0) \neq 0$, then the function $g = f - f(z_0) e_0 \in \Hol(\A)$ satisfies $g(z_0) = 0$. Hence, $L(g) = 0$, and $L(f) = L(f(z_0)e_0) = f(z_0)$. \qedhere
    \end{itemize}
\end{proof}

Along the same lines as in \cite{ACC}, Lemma~\ref{Lemme - Formes linéaires} combined with Runge's theorem \cite{Rudin}, provides an algebraic characterisation of composition operators, whose proof is omitted.   

\begin{prop}\label{Prop - Carac comp op}
    Let $T : \Hol(\A) \to \Hol(\A)$ be a continuous linear operator. Then $T$ is a composition operator if and only if for all $n \in \Z$, $Te_n = (Te_1)^n$, where $e_n(z) = z^n$.
\end{prop}

\subsection{From distances to seminorms}

The two next results are contained in \cite{COP}.

\begin{lem}[{\cite[Lemma 1.1]{COP}}]\label{Lemme - Isométrie Fréchet}
    Let $X$ be a Fréchet space, and $(\norm{\cdot}_k)_{k \ge 1}$ an associated increasing family of semi-norms.
    A linear operator $T : X \to X$ is isometric, considering the distance $d_X$ defined by
    \[ d_X(x, y) = \sum_{k = 1}^{\infty} 2^{-k} \min(1, \norm{x-y}_{k}), \]
    if and only if for all $x \in X$ and $k \in \N$, $\norm{Tx}_k = \norm{x}_k$.
\end{lem}

\begin{lem}[{\cite[Lemma 1.2]{COP}}]
    Let $X$ be a Fréchet space, and $(\norm{\cdot}_k)_{k \ge 1}$ the associated increasing family of semi-norms.
    A linear operator $T : X \to X$ is isometric, considering the distance $d'_X$ defined by
    \[ d'_X(x, y) = \sum_{k = 1}^{\infty} 2^{-k} 
        \frac{\norm{x-y}_k}{1+\norm{x-y}_k}, \]
    if and only if for all $x \in X$ and $k \in \N$, $\norm{Tx}_k = \norm{x}_k$.
\end{lem}

In particular, we consider here the seminorms
\[ \norm{f}_{\infty, n} := \sup_{z \in K_n} \abs{f(z)}, \quad
    K_n = \{ z \in \C : R^{\frac 1 n - 1} \le \abs{z} \le R^{1- \frac 1 n} \}. \]
Clearly, $K_1$ is the unit circle $\T$. In the following, denote $R_n = R^{1-\frac 1 n}$.


\section{Three useful characterisations}\label{Sec - Characterisations}

We first recall the famous  Hadamard three-circle theorem. Its detailed proof can be found for example in \cite{Milner}. However, this paper deals with regular functions on an annulus rather than a half-line. Since we work with holomorphic maps on the entire annulus, we can be more precise in the case of  equality.

\begin{theo}\label{Thm - Hadamard} ~

    Let $f \in \Hol(\A)$, and for $r \in (1/R, R)$, let $M(r) = \sup\{\abs{f(z)} : \abs{z} = r\}$. Then, the map $r \mapsto \log(M(r))$ is convex, i.e. for all $1/R < r_1 < r_2 < r_3 < R$,
    \begin{equation}\label{Eq - Hadamard}
        \log\left(\frac{r_3}{r_1}\right) \log(M(r_2))
        \le \log\left(\frac{r_3}{r_2}\right) \log(M(r_1))
            + \log\left(\frac{r_2}{r_1}\right) \log(M(r_3)).
    \end{equation} 

    Moreover, we have equality if and only if $f(z) = c z^n$ for some $c \in \C$ and $n \in \Z$.
\end{theo}

We can use this result to show that the analytic functions on an annulus $\mathbb A$ preserving three circles of $\mathbb A$ are the obvious ones.

\begin{coro}\label{cor - Carac rot}
    Let $f \in \Hol(\A)$, and $1/R < r_1 < r_2 < r_3 < R$. Assume that $f(r_j \T) \subset r_j \T$ for all $j = 1, 2, 3$.
    Then, there exists $c \in \T$ such that $f(z) = cz$.
\end{coro}

\begin{proof}
    Taking the notations of Theorem \ref{Thm - Hadamard}, for all $j = 1, 2, 3$, $M(r_j) = r_j$. If the map $f$ is not of the form $c z^n$, with $c \in \C$ and $n \in \Z$, then
    \begin{align*}
        & \log\left(\frac{r_3}{r_1}\right) \log(r_2)
        < \log\left(\frac{r_3}{r_2}\right) \log(r_1)
            + \log\left(\frac{r_2}{r_1}\right) \log(r_3) \\
        \iff & \log(r_3)\log(r_2) - \log(r_1)\log(r_2)
        < \log(r_2)\log(r_3) - \log(r_2)\log(r_1) \\
        \iff & 0 < 0.
    \end{align*}
    It follows that $f(z) = c z^n$ for some $c \in \C$ and $n \in \Z$. Then, $f(r_j\T) \subset \abs{c} r_j^n \T$. Thus,
    \[ r_j = \abs{c} r_j^n, \quad j = 1, 2, 3, \]
    which means that $\abs{c} = 1$ and $n = 1$.
\end{proof}

We now present and prove the last tool we use in the proof of our main result. The ideas are very similar to the ones involved in \cite[Theorem 1.4]{COP}.
 
\begin{theo}\label{Lemme - Carac Blaschke finis}
    Let $f \in \Hol(\A)$ such that for some $r \in (1/R, R)$ and $\rho > 0$,
    \[ \sup_{\abs{z} = r} \abs{f(z)} = \rho. \]
    Then the set $\Xi = \{\xi \in r\T : \abs{f(\xi)} = \rho\} = f^{-1}(\rho\T) \cap r\T$ is either equal to $r\T$, or a finite set.
\end{theo}

\begin{proof}
    Assume that $\Xi \neq r\T$. Since $f$ is continuous on $\A$ and $f^{-1}(\rho\T) \neq \varnothing$, we know that $f^{-1}(\rho\T)$ is a closed subset of $\A$. Hence, $\Xi$ is a closed subset of $r\T$. \medskip

    Assume that $\Xi$ is not finite.
     Then, the set $r\T \backslash \Xi$ is an open subset of $r\T$, which is a union of an infinite number of open circular arcs of $r\T$, and 
    one of the endpoints is a limit point of $\Xi$, denoted in the following as $z_0$. \medskip

    This point satisfies $\abs{f(z_0)} = \rho$, and there exist two sequences $(u_n) \subset \Xi$ and $(v_n) \subset r\T \backslash \Xi$ (it is sufficient to take points from the chosen arc) which tend to $z_0$, such that $\abs{f(u_n)} = \rho$ and $\abs{f(v_n)} \neq \rho$. \medskip
    
    We set $\phi_1 : r\T \to \R \cup \{\infty\}$ and $\phi_2 : \rho\T \to \R \cup \{\infty\}$ two conformal maps defined by
    \[ \phi_1(z) = i \frac{z-z_0}{z+z_0}, \qquad
        \phi_2(z) = i \frac{z-f(z_0)}{z+f(z_0)}. \]
    They satisfy $\phi_1(z_0) = 0$ and $\phi_2(f(z_0)) = 0$. Denote $h = \phi_2 \circ f \circ \phi_1^{-1}$.
    
    Since $\phi_1^{-1}$ is holomorphic near $0$, as is $\phi_2$ near $f(z_0)$, the function $h$ is holomorphic around $0$, satisfies $h(0) = 0$ and there exist two sequences $(a_n), (b_n) \subset \R$ (corresponding to $(u_n)$ and $(v_n) \subset r\T$) such that $a_n \to 0$, $b_n \to 0$, $h(a_n) \in \R$ and $h(b_n) \not\in \R$. However,
    all the derivatives of $h$ at $0$ must be real, so $h(b_n) \in \R$ (since $b_n \in \R$). Impossible. \medskip

    Finally, we must have $r\T \backslash \Xi = \varnothing$, i.e., $\Xi = r\T$, that is $f(r\T) \subset \rho\T$.
\end{proof}


\section{Main result}\label{Sec - Main result}

\begin{theo}\label{Thm - Résultat principal}
    Let $T : \Hol(\A) \to \Hol(\A)$ a linear and continuous operator such that for all $f \in \Hol(\A)$ and $n \in \N$,
    \[ \norm{T(f)}_{\infty, n} = \norm{f}_{\infty, n}
        = \sup_{z \in K_n} \abs{f(z)}. \]
    Then there exist two constants $\alpha, \beta \in \T$ such that $T = T_{\alpha, \beta}$ or $S_{\alpha, \beta}$, defined by
 \begin{equation}\label{eq:tabsab}
  T_{\alpha, \beta}(f)(z) = \alpha f(\beta z), \quad
        S_{\alpha, \beta}(f)(z) = \alpha f(\beta/z). 
        \end{equation}
\end{theo}

\begin{proof}
    The steps of the proof go along the same lines as in \cite[Theorem 2.1]{COP}, with important technical changes since our domain is not simply connected.
    In the following, denote by $e_k(z) = z^k$ and $f_k = Te_k$, for $k \in \Z$. \medskip

    \uline{Step 1}: We begin by showing that $Te_0 = \alpha e_0$, with $\abs{\alpha} = 1$.
    Indeed, since the operator $T$ is isometric for the seminorms $\norm{\cdot}_{\infty, 2}$ and $\norm{\cdot}_{\infty, 3}$, we obtain
    \[ \norm{f_0}_{\infty, 2} = \norm{f_0}_{\infty, 3} = 1. \]
    Using the maximum modulus principle, the map $f_0$ is constant and unimodular. Denote by $\alpha$ this constant. Without loss of generality, we will assume that $\alpha = 1$. \medskip

    \uline{Step 2}: We show that $f_k$ is a rotation or an inversion, for all $k \in \Z$. Note that for $n \in \N$,
    \[ \norm{f_k}_{\infty, n} = \norm{e_k}_{\infty, n} = (R_n)^k. \]
    The maximum is attained at the boundary of $K_n$, which is $R_n\T \cup R_n^{-1}\T$. Hence,
    \[ \sup_{\abs{z} \in \{R_n, 1/R_n\}} \abs{f_k(z)} = (R_n)^k. \]
    
    Now, remark that
    \[ (R_n)^k \T \subset f_k(K_n) =: L. \]
    
    Indeed, assume that there exists $\xi = (R_n)^k e^{i\theta}$ such that $\xi \not\in L$, which is a compact set. Then $\delta = d(\xi, L) > 0$, so $L \subset (R_n)^k \overline \D \backslash D(\xi, \delta)$. 
    We set $g = e_0 + e^{-i\theta}e_k$. Hence, $Tg = e_0 + e^{-i\theta}f_k$.
    Now, compare the seminorms: we have $\norm{g}_{\infty, n} = 1+(R_n)^k$ and
    \begin{equation*}
        \norm{Tg}_{\infty, n} 
        = \sup_{z \in K_n} \abs{1 + e^{-i\theta} f_k(z)} 
        = \sup_{w \in L} \abs{1 + e^{-i\theta}w},
    \end{equation*} 
    and since $\abs{w} \le (R_n)^k$, we have the equality $\norm{Tg}_{\infty, n} = 1 + (R_n)^k$ if and only if $e^{-i\theta} w = (R_n)^k$, which is impossible because $w \in L$.
    Thus, $T$ cannot be isometric for $\norm{\cdot}_{\infty, n}$, an absurdity.
    Therefore, $(R_n)^k \T \subset L$. \medskip
    
    Assume that $\Xi_1 = \{\xi \in R_n\T : \abs{f_k(\xi)} = (R_n)^k\}$ and $\Xi_2 = \{\xi \in R_n^{-1}\T : \abs{f_k(\xi)} = (R_n)^k\}$ are both finite sets.
    Since $f_k(\mathring{K_n}) \subset (R_n)^k\D$ and $f_k(\Xi_1 \cup \Xi_2)$ is finite, we obtain a contradiction because $(R_n)^k \T \subset L$. 
    Therefore, we have two choices: either $\Xi_1 = R_n\T$, or $\Xi_2 = R_n^{-1}\T$. In other words, for all $n \in \N$,
    \[ (E_1) : f_k(R_n\T) \subset (R_n)^k\T \quad \text{ or } \quad
        (E_2) : f_k(R_n^{-1}\T) \subset (R_n)^k\T. \]

    In particular, for $n = 1$, both equations give $f_k(\T) \subset \T$. 
    Moreover, there are at least two different values of $n$ among $2, 3$ and $4$ satisfying both $(E_1)$, or both $(E_2)$.
    Without loss of generality, we assume that $n = 2$ and $n = 3$ satisfy the same inclusion.
    \begin{itemize}[label=$\star$]
        \item If $f_k(R_n\T) \subset (R_n)^k\T$ and $g_k(z) = f_k(z)/z^{k-1}$, then $g_k(R_n\T) \subset R_n\T$ for $n = 1, 2, 3$.
        Hence, by Corollary \ref{cor - Carac rot}, $g_k(z) = c_k z$, with $\abs{c_k} = 1$, so $f_k(z) = c_k z^k$. \smallskip

        \item If $f_k(R_n^{-1}\T) \subset (R_n)^k\T$ and $g_k(z) = z^{k+1} f_k(z)$, then $g_k(R_n^{-1}\T) \subset R_n^{-1}\T$ for $n = 1, 2, 3$.
        Hence, by Corollary \ref{cor - Carac rot}, $g_k(z) = c_k z$, with $\abs{c_k} = 1$, so $f_k(z) = c_k/z^k$.
    \end{itemize}

    Thus, for all $k \in \Z$, $f_k = c_k e_k$ or $c_k e_{-k}$, for some $c_k \in \T$. \medskip

    \uline{Step 3}: Assume that $f_1 = c_1 e_1$. If $f_k = c_k e_{-k}$ for some $k \in \Z$, then for all $n \in \N$,
    \[ R_n + (R_n)^k = \norm{e_1 + e_k}_{\infty, n}
        = \norm{f_1 + f_k}_{\infty, n} 
        = \norm{c_1 e_1 + c_k e_{-k}}_{\infty, n}. \]
    Note that
    \[ \norm{c_1 e_1 + c_k e_{-k}}_{\infty, n}
        = \sup_{z \in K_n} \abs{c_1 z + \frac{c_k}{z^k}}
        = \sup_{z \in K_n} \abs{z + \frac{c_k}{c_1} \frac{1}{z^k}}. \]
        
    By the maximum modulus principle, the maximum is attained at the boundary of $K_n$, which is $R_n\T \cup R_n^{-1}\T$. We consider two possibilities:
    \begin{itemize}[label=$\star$]
        \item If the maximum is attained in $R_n\T$, then
        \begin{align*}
             R_n + (R_n)^k 
            & = \sup_{\abs{z} = R_n} \abs{z + \frac{c_k}{c_1} \frac{1}{z^k}} \\
            & = \frac{1}{(R_n)^k} \sup_{\abs{z} = R_n} 
                \abs{z^{k+1} + \frac{c_k}{c_1}}
            = \frac{(R_n)^{k+1} + 1}{(R_n)^k} = R_n + \frac{1}{(R_n)^k}.
        \end{align*}

        \item If the maximum is attained in $R_n^{-1}\T$, then
        \begin{align*}
            R_n + (R_n)^k 
            & = \sup_{\abs{z} = 1/R_n} \abs{z + \frac{c_k}{c_1} \frac{1}{z^k}} \\
            & = (R_n)^k \sup_{\abs{z} = 1/R_n} \abs{z^{k+1} + \frac{c_k}{c_1}} \\
            & = (R_n)^k \left(\frac{1}{(R_n)^{k+1}} + 1\right)
            = \frac{1}{R_n} + (R_n)^k.
        \end{align*}
    \end{itemize}

    Both of the cases lead to a contradiction, since $R_n \neq 1$ for $n \ge 2$. Hence, $f_k = c_k e_k$ for some $\abs{c_k} = 1$. In the same way, if $f_1 = c_1 e_{-1}$, then $f_k = c_k e_{-k}$ for all $k \in \Z$. \medskip

    \uline{Step 4}: In the following, we assume that $f_1 = c_1 e_1$.
    Let $k \in \NN$. Note that
    \begin{align*} 
        \sup_{z \in \partial K_n} \abs{(f_k + f_{k+1} + f_{k+2})(z)}
        & = \sup_{z \in \partial K_n} \abs{(c_k e_k 
            + c_{k+1}e_{k+1} + c_{k+2}e_{k+2})(z)} \\
        & = \sup_{z \in \partial K_n} \abs{c_k z^k 
            + c_{k+1}z^{k+1} + c_{k+2} z^{k+2}} \\
        & = \sup_{z \in \partial K_n} \abs{z^k + z^{k+1} + z^{k+2}} \\
        & = (R_n)^k + (R_n)^{k+1} + (R_n)^{k+2}.
    \end{align*}
     
    Recall that $\partial K_n = R_n\T \cup R_n^{-1}\T$. Hence, for $n \ge 2$ and $k \in \N$,
    \begin{align*}
        \sup_{\abs{z} = 1/R_n} \abs{c_k z^k 
            + c_{k+1}z^{k+1} + c_{k+2} z^{k+2}}
        & \le \frac{1}{(R_n)^k} + \frac{1}{(R_n)^{k+1}} + \frac{1}{(R_n)^{k+2}} \\
        & < (R_n)^k + (R_n)^{k+1} + (R_n)^{k+2}.
    \end{align*}
    In the same way,
    \begin{align*}
        \sup_{\abs{z} = R_n} \abs{c_k z^{-k} 
            + c_{k+1}z^{-k-1} + c_{k+2} z^{-k-2}}
        & \le \frac{1}{(R_n)^k} + \frac{1}{(R_n)^{k+1}} + \frac{1}{(R_n)^{k+2}} \\
        & < (R_n)^k + (R_n)^{k+1} + (R_n)^{k+2}.
    \end{align*}
    Thus,
    \begin{align*}
        & (R_n)^k + (R_n)^{k+1} + (R_n)^{k+2} \\
        = \; &
        \begin{cases}
        \sup_{\abs{z} = R_n} \abs{c_k z^k + c_{k+1}z^{k+1} + c_{k+2} z^{k+2}},
            & k \in \N \\
        \sup_{\abs{z} = 1/R_n} \abs{c_k z^{-k} + c_{k+1}z^{-k-1} 
            + c_{k+2} z^{-k-2}},& k \in \N.
        \end{cases}
    \end{align*}
    Up to considering $w = 1/z$, we only have to consider $k$ positive.
    Dividing by $z^k$, we have
    \[ \sup_{\abs{z} = R_n} \abs{c_k + c_{k+1}z + c_{k+2}z^2}
        = 1 + R_n + (R_n)^2. \]
    Dividing by $c_k$, since $\abs{c_k} = 1$, we have
    \[ \sup_{\abs{z} = R_n} \abs{1 + \frac{c_{k+1}}{c_k} z 
            + \frac{c_{k+2}}{c_k} z^2} = 1 + R_n + (R_n)^2. \]
    Using the triangle inequality, for all $z$ satisfying $\abs{z} = R_n$,
    \[ \abs{1+\frac{c_{k+1}}{c_k}z+\frac{c_{k+2}}{c_k}z^2}
        \le 1 + R_n + (R_n)^2, \]
    with equality if and only if
    \[ 0 = \arg(1) \equiv \arg\left[\frac{c_{k+1}}{c_k}z\right] 
        \equiv \arg\left[\frac{c_{k+2}}{c_k}z^2\right] \mod 2\pi. \]
    Denoting $\theta = \arg[z]$, this gives
    \[ 0 \equiv \arg[c_{k+1}] - \arg[c_k] + \theta 
        \equiv \arg[c_{k+2}] - \arg[c_k] + 2\theta \mod 2\pi. \]
    Finally,
    \begin{align*}
        & \arg{[c_k]} - 2\arg{[c_{k+1}]} + \arg{[c_{k+2}]} \\
        = \; & (\arg{[c_{k+2}]} - \arg{[c_k]} + 2\theta)
            - 2(\arg{[c_{k+1}]} - \arg{[c_k]} + \theta)
        \equiv 0 \mod 2\pi.
    \end{align*}
    We show by induction that for all $k \in \N_0$,
    $\arg[c_k] \equiv k\arg[c_1] \mod 2\pi$.
    \begin{itemize}[label=\sbt]
        \item $k = 0$ or $1$: It is immediate (see Step 1 for $k = 0$). \medskip

        \item Assume that the formula is valid for some nonnegative integers $k$ and $k+1$. Since $\arg{[c_k]} - 2\arg{[c_{k+1}]} + \arg{[c_{k+2}]} \equiv 0 \mod 2\pi$,
        using induction hypothesis, we have 
        \[ k\arg[c_1] - 2(k+1)\arg[c_1] + \arg[c_{k+2}] \equiv 0 \mod 2\pi. \]
        After rearranging the terms, we obtain $\arg[c_{k+2}] \equiv (k+2)\arg[c_1] \mod 2\pi$.
    \end{itemize}
    In the same way, for all $k \in \N$, $\arg[c_{-k}] \equiv k\arg[c_{-1}] \mod 2\pi$. \medskip

    \uline{Step 5}: We have shown that under the assumption $f_1 = c_1 e_1$, for some $\abs{c_1} = 1$,
    \begin{itemize}[label=$\star$]
        \item For all $k \in \Z$, $f_k = c_k e_k$ (Step 3),
        \item For all $k \in \N$, $c_k = c_1^k$ and $c_{-k} = c_{-1}^k$ (Step 4).
    \end{itemize}
    The only thing left to show is that $c_{-1} = 1/c_1$. Indeed,
    \begin{align*} 
        \sup_{z \in \partial K_n} \abs{(f_{-1} + f_0 + f_1)(z)}
        & = \sup_{z \in \partial K_n} \abs{(c_{-1} e_{-1} + e_0 + c_1 e_1)(z)} \\
        & = \sup_{z \in \partial K_n} \abs{\frac{c_{-1}}{z} + 1 + c_1 z} \\
        & = \sup_{z \in \partial K_n} \abs{\frac{1}{z} + 1 + z}
        = \frac{1}{R_n} + 1 + R_n.
    \end{align*}
    We have two choices. If the maximum is attained on $R_n\T$, then
    \[ \sup_{\abs{z} = R_n} \abs{\frac{c_{-1}}{z} + 1 + c_1 z}
        = \frac{1}{R_n} + 1 + R_n. \]
    However, we have equality if and only if $\arg(c_{-1}/z) \equiv \arg(1) \equiv \arg(c_1 z) \mod 2\pi$, that is
    \[ 0 \equiv \arg(c_{-1}) - \arg(z) \equiv \arg(c_1) + \arg(z) \mod 2\pi. \]
    Hence, $\arg(z) \equiv \arg(c_{-1}) \equiv -\arg(c_1) \mod 2\pi$, so
    $c_{-1} = 1/c_1$. It is the same if the maximum is attained on $R_n^{-1}\T$. \medskip

    \fbox{Conclusion:}  For all $k \in \Z$, $f_k = c_1^k e_k = (c_1 e_1)^k = f_1^k$. Hence, $T$ is a composition operator, with symbol $\varphi = f_1 = c_1 e_1$ or $c_1 e_{-1}$, by Proposition \ref{Prop - Carac comp op}.
    Finally, multiplying by the $\alpha$ of Step 1, for all $f \in \Hol(\A)$,
    \[ (Tf)(z) = \alpha f(c_1 z) \text{ or } \alpha f(c_1/z). \qedhere \]
\end{proof}

\begin{rem}
    The proof of the main theorem only needs three different seminorms: $\norm \cdot_{\infty, 1}$, and two among $\norm \cdot_{\infty, 2}$, $\norm \cdot_{\infty, 3}$ and $\norm \cdot_{\infty, 4}$. In other words, we can rewrite the main theorem with weaker assumptions:

    Let $T : \Hol(\A) \to \Hol(\A)$ a linear and continuous operator such that for all $f \in \Hol(\A)$ and $n \in \{1, 2, 3, 4\}$,
    \[ \norm{T(f)}_{\infty, n} = \norm{f}_{\infty, n}
        = \sup_{z \in K_n} \abs{f(z)}. \]
    Then there exist two constants $\alpha, \beta \in \T$ such that $T = T_{\alpha, \beta}$ or $S_{\alpha, \beta}$.
\end{rem}

\section{Description of unimodular functions on the unit circle}\label{Sec - Unimodular maps unit circle}

There exist many approximation results for measurable unimodular functions on the unit circle such as 
the Helson--Sarason Theorem \cite{HS67} which says that given $\varepsilon > 0$ and $f$ unimodular and continuous on $\T$, there exist finite Blaschke products $B_1$ and $B_2$ such that 
\[ \norm{f-\frac{B_1}{B_2}}_\infty < \varepsilon.\]
In this section we investigate factorisation results instead of approximation results for unimodular function on the unit circle. The first one is due to Fatou.   
    
\begin{theo}[\cite{Fatou23}]
    If $f$ is analytic on $\D$ and $\lim_{|z|\to 1}|f(z)| = 1$,
    then $f$ is a finite Blaschke product. 
    In particular, if $f$ is a function in the disc algebra $A(\D)$, then $f$ is a finite Blaschke product whenever $f$ is unimodular on $\T$.
\end{theo}
 
There is also a meromorphic version of this.

\begin{coro}[\cite{GMR17}]\label{cor:mero}
    Suppose that $f$ is meromorphic on $\D$ and extends continuously to $\T$, the unit circle. If $f$ is unimodular on $\T$, then $f$ is a quotient of two finite Blaschke products.
\end{coro}

In 1986, Bourgain solved a problem of factorization posed by Douglas and Rudin.
\begin{theo}[\cite{Bourgain86,Barclay}]\label{th:bourgain}
    Let $f$ be a bounded measurable function on the unit circle $\T$.
    Then the following assertions are equivalent: 
    \begin{enumerate}
        \item The function $\log |f|$ belongs to $L^1(\T)$,
        \item There exist $g, h \in H^\infty$ such that $f = g \overline{h}$. 
    \end{enumerate}
     
    In particular, any measurable function unimodular almost everywhere on the unit circle is of the form $f=g\overline{h}$, for some functions $g, h \in H^\infty$.
\end{theo}

The following theorem is a complete description of unimodular functions on the unit circle that are analytic in a neighbourhood of $\T$.

\begin{theo}\label{th:newfacto}
    Let $f$ be an analytic map on $\A$. Then the following assertions are equivalent:
    \begin{enumerate}
        \item The function $f$ is unimodular on $\T$,
        \item there exists $s \in (1, R]$, $n \in \Z$ and an outer function $g_0^s$ in $H^\infty$ which extends analytically on $s\D$ such that 
        \[ f(z) = z^n
            \frac{g_0^s(z)}{\overline{{g_0^s(1/\overline{z})}}} \quad
            \text{ for all } z \in \A_s := \{z \in \C : 1/s < \abs{z} < s\}. \]     	
    \end{enumerate}	 
\end{theo}

For example, a Blaschke factor $f(z)=(z-\alpha)/(1-\overline\alpha z)$ with $|\alpha|<1$ arises on taking
$g_0^s(z)=1/(1-\overline\alpha z)$ with $n=1$ and $1<s<1/|\alpha|$. \medskip

Here is an example showing that the set of unimodular functions on $\T$ and analytic in a neighborhood of $\T$ is much larger than the set of quotients of finite Blaschke products.

Let $f(z) = \exp(z-1/z)$. Then $f(\xi) = \exp(\xi-\overline{\xi}) = \exp(2i\Im (\xi))$ for all $\xi \in \T$. 
It follows that $f$ is unimodular on $\T$ and analytic on $\C^*$. 
In this case we have $s = R$, $n = 0$ and $g_0(z) = \exp(z)$. \medskip

Our proof relies on Theorem \ref{th:bourgain} and the following lemma. 
\begin{lem}\label{lem:key1}
    Let $s > 1$. Suppose that $g$ is analytic and does not vanish in the annulus $\A_s = \{z \in \C : 1/s < \abs z < s\}$.
    Then there exist $n \in \Z$ and $h$ analytic in $\A_s$ such that 
    \[ g(z) = z^n \exp(h(z)) \quad \text{ for all } z \in \A_s. \]
\end{lem}
 
\begin{proof}
    Since $\A_s$ is not simply connected, the existence of an analytic map $h$ in $\A_s$ such that $g(z) = \exp(h(z))$ is not necessarily true. 
    Now recall that the existence of $h$ such that $g = \exp(h)$ on a domain $\Omega$ is equivalent to 
    \[ \int_C \frac{g'(z)}{g(z)} \dd z = 0 \quad
        \text{ for all Jordan curves } C \text{ in } \Omega. \]
    For $\Omega = \A_s$, this is equivalent to 
    \[ \int_{C_t} \frac{g'(z)}{g(z)} \dd z = 0 \quad
        \text{ for all } t \in (1/s,s), 
        \text{ where } C_t = t\T.\]
     Note that $\int_{C_t} g'(z)/g(z) \; \dd z \in 2i\pi\Z$, and therefore there exists $n\in \Z$ such that 
     \[ \int_{C_t} \frac{g'(z)}{g(z)} \dd z = 2i\pi n. \]
     It follows that $\tilde g (z) := g(z)/z^n$ satisfies 
      \[ \int_{C_t} \frac{\tilde{g}'(z)}{\tilde{g}(z)} \dd z = 0 \quad
            \text{ for all } t \in (1/s,s). \]
     Finally, $g(z) = z^n \exp (h(z))$ for all $z \in \A_s$.
\end{proof}

\begin{proof}[Proof of Theorem \ref{th:newfacto}]
    Since $f$ is analytic in $\A$ and unimodular on the unit circle, there exists $s \in (1, R]$ such that $f(z) \neq 0$ for all $z \in \A_s$.
    By Lemma \ref{lem:key1}, there exist $n \in \Z$ and $h$ analytic in $\A_s$ such that
    \[ f(z) = z^n \exp(h(z)) \text{ for all } z \in \A_s. \]
    
    Note that $h(z) = h_+(z) + h_{-}(z)$, where $h_+$ is analytic on $s\D$, $h_{-}$ is analytic on $\C \backslash s^{-1}\D$ and $h_+(0) = h(0)$. 
    It follows that $g_0 := \exp(h_+)$ is an outer function which extends analytically on $s\D$. Indeed, $\exp(h_+)$ has no zero in $\D$ and moreover, since it extends analytically on $s\D$, its singular inner part is trivial. In addition,
    \[ g_1(z) := \begin{cases}
        \overline{\exp(h_-(1/\overline z))}, & 0 < \abs{z} < s \\
        1, & z = 0
    \end{cases} \]
    is also analytic on $s\D$ and outer. Therefore, we have
    \[ f(z) = z^n g_0(z) \overline{g_1(1/\overline z)}, \quad z \in s\D. \]

    Since $f$ is unimodular on $\T$, we get $\abs{g_0(\xi)}\abs{g_1(\xi)} = 1$ for all $\xi \in \T$. It follows that the outer functions $g_0$ and $g_1$ satisfy $g_1(z) = 1/g_0(z)$, so for all $z \in \A_s$,
    \[ f(z) = z^n \frac{g_0(z)}{\overline{{g_0(1/\overline{z})}}}. \]

    The reverse implication is immediate, since for all $z \in \T$, $1/\overline z = z$.
\end{proof}

\section{Spectral study of linear isometries on the annulus}\label{Sec - Spectrum}

In this section we consider the spectrum $\sigma$ and point spectrum $\sigma_p$ of the operators
$T_{\alpha,\beta}$ and $S_{\alpha,\beta}$ defined in \eqref{eq:tabsab}.
It will be convenient to begin with the simplest of these, namely $S_{\alpha,\beta}$,
where $S_{\alpha,\beta}f(z)=\alpha f(\beta/z)$.

\begin{theo}
For $|\alpha|=|\beta|=1$ we have $\sigma(S_{\alpha,\beta})=\sigma_p(S_{\alpha,\beta})=\{\alpha,-\alpha\}$. Moreover, $\ker (S_{\alpha,\beta}-\lambda \Id)$ is infinite-dimensional
for $\lambda \in \{\alpha,-\alpha\}$.
\end{theo}

\begin{proof}
For $f(z)=z^k \pm \beta^k/z^k$ with $k \in \NN$ we have 
\[
S_{\alpha,\beta}f(z)=  \alpha \beta^k/z^k \pm \alpha \beta^k/(\beta^k/z^k)= \pm\alpha f(z),
\]
 and so we see that $\{ \alpha, -\alpha\} \subseteq \sigma_p(S_{\alpha,\beta})$
 and   the eigenspaces are infinite-dimensional.
 
 Moreover, $S_{\alpha,\beta}^2 = \alpha^2 \Id$, and so
 \[
 (S_{\alpha,\beta}-\lambda\Id)(S_{\alpha,\beta}+\lambda \Id)= S_{\alpha,\beta}^2 - \lambda^2 \Id=(\alpha^2-\lambda^2)\Id,
 \]
 and thus $S_{\alpha,\beta}-\lambda \Id$ is invertible if $\lambda^2 \ne \alpha^2$.
 That is, \[ \sigma_p(S_{\alpha,\beta})=\sigma(S_{\alpha,\beta})=\{\alpha,-\alpha\}. \qedhere \]
\end{proof}

We now turn our attention to $T_{\alpha,\beta}$.

\begin{theo}\label{thm:spectrum}
    Let $\alpha, \beta \in \T$ and $T_{\alpha, \beta} : \Hol(\A) \to \Hol(\A)$ defined by
    \[ T_{\alpha, \beta}(f)(z) = \alpha f(\beta z), \quad f \in \Hol(\A). \]

    \begin{enumerate}
        \item If there exists $n \in \N$ such that $\beta^n = 1$, then
        \[ \sigma_p(T_{\alpha, \beta}) = \sigma(T_{\alpha, \beta})
            = \{\alpha \beta^k : 0 \le k \le n-1\}. \]
        Moreover, for all $\lambda \in \sigma_p(T_{\alpha, \beta})$, $\Ker(T_{\alpha, \beta} - \lambda \Id)$ has infinite dimension.

        \item If for all $n \in \Z$, $\beta^n \neq 1$, then
        \[ \sigma_p(T_{\alpha, \beta}) 
            = \{\alpha \beta^k : k \in \Z\}
            \subset \sigma(T_{\alpha, \beta}) \subset \T. \]
        Moreover, for all $\lambda \in \sigma_p(T_{\alpha, \beta})$, $\Ker(T_{\alpha, \beta} - \lambda \Id)$ has dimension 1. 
    \end{enumerate}
\end{theo}

\begin{proof} \;
    \begin{enumerate}
        \item Let $e_k(z) = z^k$, $0 \le k \le n-1$. Then,
        \[ T_{\alpha, \beta}(e_k)(z) = \alpha e_k(\beta z) 
                = \alpha (\beta z)^k = \alpha \beta^k z^k = \alpha \beta^k e_k(z). \]
        Therefore, $\alpha \beta^k \in \sigma_p(T_{\alpha, \beta})$, and $e_k \in \Ker(T_{\alpha, \beta} - \alpha \beta^k \Id)$. In addition, it is easy to show that for all $\ell \in \Z$, $e_{k+n\ell} \in \Ker(T_{\alpha, \beta} - \alpha \beta^k \Id)$, since $\beta^{k+n\ell} = \beta^k$, so the eigenspaces have infinite dimension. \medskip

        Moreover, note that $T_{\alpha, \beta}^n = \alpha^n \Id$, so
        \[ (\alpha^n - \lambda^n) \Id = T_{\alpha, \beta}^n - \lambda^n \Id 
            = (T_{\alpha, \beta} - \lambda \Id) 
                \sum_{k=1}^n \lambda^{k-1} T_{\alpha, \beta}^{n-k}. \]
        Hence, $T_{\alpha, \beta} - \lambda \Id$ is invertible if and only if $\lambda^n \neq \alpha^n$, i.e. $\lambda \not\in \{\alpha \beta^k : 0 \le k \le n-1\}$ since $\beta^n = 1$. \medskip
        
        \item Let $f \in \Hol(\A)$. Note that we can write
        \[ f(z) = \sum_{k \in \Z} a_k z^k 
                = a_0 + \sum_{k \ge 1} a_k z^k + \sum_{k \ge 1} a_{-k} z^{-k}, \]
        where the first series converges for $|z|<R$ and the second one converges for $|z|>1/R$. Then,
        \[ (T_{\alpha, \beta} - \lambda \Id)(f)(z) = 
            \sum_{k \in \Z} (\alpha \beta^k - \lambda) a_k z^k. \]
        
        Let $g = \sum_{k \in \Z} b_k z^k \in \Hol(\A)$. If $\lambda \not \in \T$, then consider, for $k \in \Z$, $a_k = \frac{b_k}{\alpha \beta^k - \lambda}$.
        Thus,
        \[ \frac{\abs{b_k}}{1 + \abs{\lambda}} < \abs{a_k} < \frac{\abs{b_k}}{\abs{1 - \abs{\lambda}}}. \]
        Hence, the power series determined by the sequences $(a_k)$, $(b_k)$, $(a_{-k})$ and $(b_{-k})$ have again the same radius of convergence, namely $R$. This means that $f = \sum_{k \in \ZZ} a_k z^k \in \Hol(\A)$. Moreover, by construction, $(T_{\alpha, \beta} - \lambda \Id)(f) = g$, so $(T_{\alpha, \beta} - \lambda \Id)$ is invertible: $\lambda \not\in \sigma(T_{\alpha, \beta})$. \medskip

        \noindent Now, let $e_k(z) = z^k$, $k \in \Z$. Then, $T_{\alpha, \beta}(e_k)(z) = \alpha e_k(\beta z) = \alpha \beta^k e_k(z)$, so we have $\alpha \beta^k \in \sigma_p(T_{\alpha, \beta})$, and $e_k \in \Ker(T_{\alpha, \beta} - \alpha \beta^k \Id)$. \medskip
        
        \noindent Let $h \in \Ker(T_{\alpha, \beta} - \alpha \beta^k \Id)$. Then, 
        \[ \alpha h(\beta z) = \alpha \beta^k h(z) \iff h(\beta z) = \beta^k h(z). \]
        This implies that $h(z) \neq 0$ for all $z \in \A$. Indeed, if $h(z_0) = 0$ for some $z_0 \in \A$, then for all $k \in \N$, $h(\beta^k z_0) = 0$. Since $\{\beta^k z_0 : k \in \N\}$ is a dense subset of $\abs{z_0}\T$, we obtain $h \equiv 0$. \medskip

        \noindent By Lemma \ref{lem:key1}, $h(z) = z^j \exp(\theta(z))$ for some $j \in \Z$ and $\theta \in \Hol(\A)$. 
        Denote $\psi = \exp(\theta)$. Then,
        \begin{align*}
            h(\beta z) = \beta^k h(z)
            & \iff \beta^j z^j \psi(\beta z) = \beta^k z^j \psi(z) \\
            & \iff \psi(\beta z) = \beta^{k-j} \psi(z).
        \end{align*} 
        Let $\sum c_n z^n$ be the Laurent series associated with $\psi$. Then,
        \[ \sum_{n \in \Z} c_n \beta^n z^n = \sum_{n \in \Z} c_n \beta^{k-j} z^n. \]
        Since the Laurent series is unique, we obtain $c_n = 0$ if $n \neq k-j$. Hence, $\psi(z) = c_{k-j} z^{k-j}$ for some $c_{k-j} \in \C$. However, if $k-j \neq 0$, then for all $r \in (1/R, R)$,
        \[ \int_{r\T} \frac{\psi'(z)}{\psi(z)} \dd z
            = \int_0^{2\pi} \frac{c_{k-j} (k-j) (re^{it})^{k-j-1}}
                {c_{k-j} (re^{it})^{k-j}} i re^{it} \dd t
            = 2i\pi(k-j) \neq 0. \]
        Therefore, in this case, $\psi$ cannot be written as $\exp(\theta)$. We conclude that $j=k$, so $\psi$ is constant, and $h \in \spam(e_k)$. This proves that $\Ker(T_{\alpha, \beta} - \alpha \beta^k \Id)$ has dimension 1. \qedhere
    \end{enumerate}
\end{proof}

The spectrum of $T_{\alpha, \beta}$ can be different from $\T$ if $\beta$ is aperiodic. We take our inspiration from the results on the disc \cite{ABBC-rot}. 

\begin{expl}
    Let $\tau > 2$. The set of \emph{Diophantine numbers of order $\tau$} is defined by
    \[ \mathcal D (\tau) = \left\{\xi \in \R : \exists \gamma > 0, \forall p \in \Z, \forall q \in \N, \abs{\frac p q - \xi} \ge \frac \gamma {q^{\tau}} \right\} \]

    Let $\xi \in \mathcal D (\tau)$, and $\beta = e^{2i\pi\xi}$. Then, we show that for all $r \in \Q \backslash \Z$, $\lambda = e^{2i\pi r} \not\in \sigma(T_{1, \beta})$. Since multiplying by $\alpha$ only rotates our problem, this is also true for $T_{\alpha, \beta}$, $\alpha \in \T$. \medskip

    Indeed, let $\gamma > 0$ such that $\abs{\xi - p/q} \ge \gamma q^{-\tau}$, for all $p \in \Z$ and $q \in \N$. Let $r = p_0/q_0$ and $k \in \N$. Choose $p \in \Z$ such that $\abs{\xi k - r - p} \le 1/2$. Then,
    \begin{align*}
        \abs{\beta^k - \lambda} & = \abs{e^{2i\pi\xi k} - e^{2i\pi r}} \\
            & = \abs{e^{2i\pi\xi k} - e^{2i\pi (r+p)}} \\
            & = \abs{e^{i\pi(\xi k - r - p)} - e^{-i\pi(\xi k - r - p)}} \\
            & = 2 \sin \abs{\pi(\xi k - r - p)}
            \ge 2 \cdot \frac 2 \pi \cdot \pi \abs{\xi k - r - p},
    \end{align*}
    since the sine function is concave on $[0, \pi/2]$ and $\abs{\pi(\xi k - r - p)} \in [0, \pi/2]$. Consequently,
    \[ \frac{1}{\abs{\beta^k - \lambda}}
        \le \frac{1}{4 \abs{\xi k - r - p}}
        = \frac{1}{4k} \frac{1}{\abs{\xi - \frac{p_0 + pq_0}{q_0 k}}}
        \le \frac{1}{4k} \frac{(q_0 k)^\tau}{\gamma}
        = c k^{\tau-1}, \;\; c := \frac{q_0^\tau}{4\gamma}. \]
        
    Finally,
    \begin{align*}
        1 = \lim_{k \to +\infty} \frac{1}{2^{1/k}}
        & \le \liminf_{k \to +\infty} \frac{1}{\abs{\beta^k - \lambda}^{1/k}} \\
        & \le \limsup_{k \to +\infty} \frac{1}{\abs{\beta^k - \lambda}^{1/k}}
        \le \limsup_{k \to +\infty} c^{1/k} k^{(\tau-1)/k} \le 1.
    \end{align*} 

    Hence, if $g(z) = \sum_{n \in \ZZ} a_n z^n \in \Hol(\A)$, then
    \[ f(z) := \sum_{n \in \Z} \frac{a_n}{\beta^n - \lambda} z^n \in \Hol(\A), \]
    and $f(\beta z) - \lambda f(z) = g(z)$. Finally, $T_{1, \beta} - \lambda \Id$ is surjective, and injective by Theorem \ref{thm:spectrum}. Thus, $\sigma(T_{1, \beta}) \subset \{e^{2i\pi s} : s \not\in \Q\} \cup \{1\}$.
\end{expl}

Likewise, it is possible for $\sigma_p(\Talbe)$ to be strictly contained in $\sigma(\Talbe)$.

\begin{expl}
    For our example, we take $\alpha = 1$.

    Let $p_1 = 1$ and for $n \ge 1$, $p_{n+1} = 2^{n^{p_n}}$.
    Then, the map $g$ defined by $g(z) = \sum_{n=1}^\infty (z/R)^{p_n}$ is in $\Hol(R\D) \subset \Hol(\A)$, where $\A = \{1/R < \abs{z} < R\}$. Moreover,
    \begin{itemize}[label=$\star$]
        \item For all $n \ge 1$, $p_{n+1} = 2^{n^{p_n}} \ge 2^{p_n} \ge p_n$.
        \item For all $n \ge 1$, $p_n$ is a power of two, hence an integer.
        \item For all $n \ge 1$, denote $p_n = 2^{q_n}$.
        Then, $q_{n+1} = n^{p_n} \ge p_n \ge q_n$, so
        \[ \frac{p_{n+1}}{p_n} = 2^{q_{n+1} - q_n} \in \N. \]
        Moreover, $q_{n+1} - q_n = n^{p_n} - q_n \ge n p_n - q_n \ge (n-1) p_n \ge n-1$, so $p_{n+1}/p_n \ge 2^{n-1}$.
    \end{itemize}
    Now we take 
    \[ \beta = \exp (i\pi\theta), \quad \hbox{where} \quad
        \theta = \sum_{n=1}^\infty \frac{1}{p_n}. \]
    Note that for all $n \in \N$, we may write
    \[ \theta = \frac{\ell_n}{p_n} + \varepsilon_n, \]
    where $\ell_n$ is an integer and
    \[ 0 < \varepsilon_n = \sum_{k \ge n+1} \frac{1}{p_k}
        = \frac{1}{p_{n+1}} \sum_{k \ge n+1} \frac{p_{n+1}}{p_k}
        \le \frac{1}{p_{n+1}} \sum_{k \ge n+1} \frac{1}{2^{k-2}}
        \le \frac{2}{p_{n+1}}. \]
    In addition, $p_{n+1} \ge p_n^{n+1}$ (this is immediate for $n = 1, 2$; when $n \ge 3$, it comes from the fact that if $p_n = 2^{q_n}$, then $(n+1) q_n \le (n+1) p_n \le n^{p_n} = q_{n+1}$). Therefore
    \[ 0 < \varepsilon_n \le \frac{2}{p_n^{n+1}} \le \frac{1}{p_n^n}. \]
    This means that $\theta$ is a Liouville (hence, irrational) number.
    Thus, we have $-1 \not \in \sigma_p(T_{\alpha, \beta})$. \medskip

    We claim that $-1 \in \sigma(T_{\alpha, \beta})$.
    To solve the equation $f(\beta z) + f(z) = g(z)$, with $g(z)=\sum_{k \ge 1} (z/R)^k$, we require
    $f(z) = \sum_{k \ge 1} \frac{1}{1+\beta^k} (z/R)^k$. However,
    \[ \beta^{p_n} = \exp(i\pi p_n \theta)
        = \exp\left( i\pi \sum_{k \ge n} \frac{p_n}{p_k} \right)
        = -\exp \left(i\pi \sum_{k \ge n+1} \frac{p_n}{p_k}\right), \]
    so we obtain
    \begin{align*}
        \abs{1 + \beta^{p_n}}
        = 2 \abs{\sin\left(\pi \sum_{k \ge n+1} \frac{p_n}{p_k}\right)}
        & \underset{n \to \infty}{\sim} 2 \pi \sum_{k \ge n+1} \frac{p_n}{p_k} \\
        & \underset{n \to \infty}{\sim} 2 \pi \frac{p_n}{p_{n+1}}
        \underset{n \to \infty}{\sim} 2 \pi \frac{p_n}{2^{n^{p_n}}}.
    \end{align*} 
    Looking at the coefficient $c_n = \frac{1}{\beta^{p_n}+1}$ of $(z/R)^{p_n}$ in $f$, we see that 
    \[ |c_n|^{1/p_n} \underset{n \to \infty}{\sim} 
            \frac{2^{\frac{1}{p_n} n^{p_n}}}{p_n^{1/p_n}}
        \xrightarrow[n \to \infty]{} \infty. \]
    The radius of convergence of the series for $f$ is then $0$, which implies that $-1$ is in the spectrum. Note that if we take $R = 1$ in the above estimations, then we obtain an example of composition operator induced by a rotation on the disc for which the point spectrum is different from the spectrum, thus providing an example that complements the results of \cite{ABBC-rot}.
\end{expl}

\end{document}